\documentclass[12pt]{article}

\usepackage[usenames]{color}
\usepackage{amssymb}

\usepackage{graphicx}
\usepackage{amscd}

\usepackage[colorlinks=true,
linkcolor=webgreen,
filecolor=webbrown,
citecolor=webgreen]{hyperref}

\definecolor{webgreen}{rgb}{0,.5,0}
\definecolor{webbrown}{rgb}{.6,0,0}

\usepackage{color}
\usepackage{fullpage}
\usepackage{float}

\usepackage{graphics,amsmath,amssymb}
\usepackage{amsthm}
\usepackage{amsfonts}
\usepackage{latexsym}
\usepackage{epsf}

\theoremstyle{plain}
\newtheorem{thm}{Theorem}
\newtheorem{lemma}{Lemma}
\newtheorem{idn}{Identity}
\newtheorem{cor}{Corollary}
\newcommand{\rb}[1]{_{\rm #1}}
\newcommand{\ch}[2]{\begin{bmatrix}#1\\ #2\end{bmatrix}}
\newcommand{\tch}[2]{[\begin{smallmatrix}#1\\ #2\end{smallmatrix}]}
\newcommand{\chb}[2]{\left\langle\begin{matrix}#1\\#2\end{matrix}\right\rangle}
\newcommand{\tchb}[2]{\langle\begin{smallmatrix}#1\\#2\end{smallmatrix}\rangle}
\newcommand{\floor}[1]{\lfloor#1\rfloor}
\newcommand{\ceil}[1]{\lceil#1\rceil}

\newcommand{\seqnum}[1]{\href{http://oeis.org/#1}{\underline{#1}}}

\begin{document}

\begin{center}
\vskip 1cm{\LARGE\bf
New Combinatorial Interpretations of the \\
\vskip .01in
Fibonacci Numbers Squared, Golden \\
\vskip .01in
Rectangle Numbers, and Jacobsthal  Numbers \\
\vskip .08in
Using Two Types of Tile
}
\vskip 1cm
\large
Kenneth Edwards and Michael A. Allen\footnote{Corresponding author.}\\
Physics Department\\
Faculty of Science\\
Mahidol University\\
Rama 6 Road\\
Bangkok 10400  \\
Thailand \\
\href{mailto:kenneth.edw@mahidol.ac.th}{\tt kenneth.edw@mahidol.ac.th} \\
\href{mailto:maa5652@gmail.com}{\tt maa5652@gmail.com} \\
\end{center}

\vskip .2 in

\begin{abstract}
We consider the tiling of an $n$-board (a board of size $n\times1$)
with squares of unit width and $(1,1)$-fence
tiles. A $(1,1)$-fence tile is composed of two unit-width square
subtiles separated by a gap of unit width. We show that the number of
ways to tile an $n$-board using unit-width squares and $(1,1)$-fence
tiles is equal to a Fibonacci number squared when $n$ is even and a
golden rectangle number (the product of two consecutive Fibonacci
numbers) when $n$ is odd. We also show that the number of tilings of
boards using $n$ such square and fence tiles is a Jacobsthal number.  Using
combinatorial techniques we prove identities involving sums of
Fibonacci and Jacobsthal numbers in a straightforward
way. Some of these identities appear to be new.  We also construct and
obtain identities for a known Pascal-like triangle (which has
alternating ones and zeros along one side) whose $(n,k)$th entry is
the number of tilings using $n$ tiles of which $k$ are fence
tiles. There is a simple relation between this triangle and the
analogous one for tilings of an $n$-board.  Connections between the
triangles and Riordan arrays are also demonstrated. With the help of the
triangles, we express the Fibonacci numbers squared, golden rectangle
numbers, and Jacobsthal numbers as double sums of products of two
binomial coefficients.
\end{abstract}

\section{Introduction}
The $(n+1)$th Fibonacci number (\seqnum{A000045}), defined by
$F_{n+1}=\delta_{n,1}+F_n+F_{n-1}$, $F_{n<1}=0$, where $\delta_{i,j}$
is 1 if $i=j$ and zero otherwise, can be interpreted as the number of
ways to tile an $n$-board (a board of size $n\times1$ composed of
$1\times1$ cells) with $1\times1$ squares (henceforth referred to
simply as squares) and $2\times1$ dominoes \cite{BCCG96,BQ=03}. More
generally, the number of ways to tile an $n$-board with all the
$r\times1$ $r$-ominoes from $r=1$ up to $r=k$ is the $k$-step (or
$k$-generalized) Fibonacci number
$F^{(k)}_{n+1}=\delta_{n,1}+F^{(k)}_{n}+F^{(k)}_{n-1}+\cdots+F^{(k)}_{n-k+1}$,
with $F^{(k)}_{n<1}=0$ \cite{BQ=03}.  Edwards \cite{Edw08} showed that
it is possible to obtain a combinatorial interpretation of the
tribonacci numbers (the 3-step Fibonacci numbers, \seqnum{A000073}) as
the number of tilings of an $n$-board using just two types of tiles,
namely, squares and $(\frac12,1)$-fence tiles.  A $(w,g)$-fence tile
is composed of two subtiles (called \textit{posts}) of size
$w\times1$ separated by a gap of size $g\times1$. We presented a
bijection between the Fibonacci numbers squared (\seqnum{A007598}) and
the tilings of an $n$-board with half-squares (i.e., $\frac12\times1$
tiles always oriented so that the shorter side is horizontal) and
$(\frac12,\frac12)$-fence tiles \cite{EA19} and this was used to
formulate combinatorial proofs of various identities
\cite{EA19,EA20}. We also identified a bijection between tiling an
$n$-board with $(\frac12,g)$-fence tiles where $g\in\{0,1,2,\ldots\}$
and strongly restricted permutations and then used it to obtain
results concerning the permutations \cite{EA15}.

Here we show that the number of ways to tile an $n$-board using
square and $(1,1)$-fence tiles is a Fibonacci number squared if $n$ is
even and a golden rectangle number (the product of two successive
Fibonacci numbers, \seqnum{A001654}) if $n$ is odd. We also consider
the number of ways to tile boards using a total of $n$ of these tiles and
refer to this as an \textit{$n$-tiling}.  We show that enumerating
$n$-tilings yields the Jacobsthal numbers
$J_{n\geq0}=0,1,1,3,5,11,21,43,85,171,\ldots$ (\seqnum{A001045})
where the $n$th Jacobsthal number is defined via
\begin{equation}\label{e:Jn}
J_{n}=\delta_{n,1}+J_{n-1}+2J_{n-2}, \quad J_{n<1}=0.
\end{equation}
We use both types of tiling to formulate straightforward combinatorial
proofs of identities involving the Fibonacci numbers
squared, golden rectangle numbers, and Jacobsthal numbers.  We also
obtain two Pascal-like triangles (one for $n$-tilings, the other for
tilings of an $n$-board) whose entries are the number of tilings with
squares and $(1,1)$-fences which use a given number of fences. A
number of properties of the triangles are derived including their
relation to Riordan arrays. Finally, the triangles are used to express
the Fibonacci numbers squared, golden rectangle numbers, and
Jacobsthal numbers as double sums of a product of two binomial
coefficients.

\section{Tiling boards with squares and fences}
When tiling a board with fences it is helpful to first determine the types of
metatile since any tiling of the board can be expressed as a tiling
using metatiles \cite{Edw08}.  A \textit{metatile} is an arrangement of tiles
that exactly covers an integral number of adjacent cells and cannot be
split into smaller metatiles \cite{Edw08,EA15}.  When tiling with
squares ($S$) and $(1,1)$-fence tiles (henceforth referred to simply
as fences or $F$), the simplest metatile is the square. To tile
adjacent cells by starting with a single fence we must fill the gap
with either a square or the post of another fence. These generate what
we will refer to as the \textit{filled fence} ($FS$) and \textit{bifence} ($FF$)
metatiles, respectively (Fig.~\ref{f:metatiles}).  The bifence clearly
has a length of 4.  The filled fence has length 3 and the square
inside will be called a \textit{captured square}. A square which is not
captured (and is therefore a metatile) is called a \textit{free square}.

\begin{figure}
\begin{center}
\includegraphics[width=4cm]{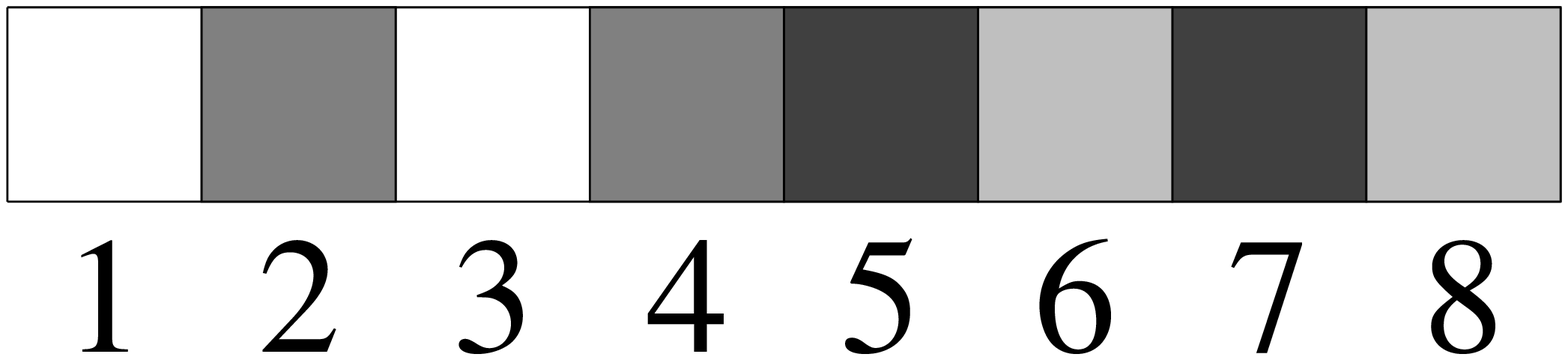}
\end{center}
\caption{An 8-board tiled with the three possible metatiles: a free
  square (cell 1), a filled fence (cells 2--4), and a bifence (cells
  5--8). The symbolic representation of this tiling is $SFSFF$.}
\label{f:metatiles}
\end{figure}

\begin{thm}\label{T:An}
Let $A_n$ be the number of ways to tile an $n$-board using squares and
fences. Then
\begin{equation}
\label{e:An}
A_{n}=\delta_{n,0}+A_{n-1}+A_{n-3}+A_{n-4}, \quad A_{n<0}=0.
\end{equation}
\end{thm}
\begin{proof}
We condition on the last metatile \cite{BHS03,EA15}. If the
last metatile is of length $l$ there will be $A_{n-l}$ ways to tile
the remaining $n-l$ cells. The result \eqref{e:An} follows from the
fact that there are three possible metatiles and these have lengths of
1, 3, and 4.  If $n=l$ there is exactly one tiling (which corresponds to
that metatile filling the entire board) so we make $A_0=1$. There is no
way to tile an $n$-board if $n<l$ and so $A_{n<0}=0$.
\end{proof}
$A_{n\ge0}=1,1,1,2,4,6,9,15,25,40,64,104,169,273,441,714,1156,\ldots$
is \seqnum{A006498}. As we will shortly prove combinatorially, the
even (odd) terms of this sequence are the Fibonacci numbers squared
\seqnum{A007598} (golden rectangle numbers \seqnum{A001654}).

\begin{lemma}\label{L:bij}
There is a bijection between the fence-square tilings of a
$2n$-board (a $(2n+1)$-board) and the square-domino tilings of an
ordered pair of $n$-boards (an $(n+1)$-board and an $n$-board).
\end{lemma}

\begin{proof}
Tile an $n$-board (an $(n+1)$-board) with the contents of the
odd-numbered cells of the given $2n$-board ($(2n+1)$-board)
fence-square tiling and tile
a second $n$-board (an $n$-board) with the contents of
the even-numbered cells. The posts of any fence (which always lie on two
consecutive odd or even cells) get mapped to a domino. The procedure
is reversed by splicing the two square-domino tilings.
\end{proof}

\begin{thm}\label{T:A2n} For $n\ge0$,
\begin{subequations}
\label{e:Aff}
\begin{align}
\label{e:A2n}
A_{2n}&=f_n^2, \\
A_{2n+1}&=f_nf_{n+1},
\label{e:A2n+1}
\end{align}
\end{subequations}
where $f_n=F_{n+1}$.
\end{thm}
\begin{proof}
There are $f_{n}$ ways to tile an $n$-board using squares and
dominoes \cite{BQ=03}. From Lemma~\ref{L:bij}, $A_{2n}$ is the same as
the number of ways to tile an ordered pair of $n$-boards using squares
and dominoes which is $f_{n}^2$, and $A_{2n+1}$ is the same as
the number of ways to tile an $n$-board and $(n+1)$-board using squares
and dominoes which is $f_nf_{n+1}$.
\end{proof}

As it is easily done and the result will be used in future work, in
the following theorem we generalize Theorem~\ref{T:A2n} to the case of
tiling an $n$-board with squares and $(1,m-1)$-fences for some fixed
$m\in\{2,3,\ldots\}$.

\begin{thm}\label{T:Amn} If $A_n^{(m)}$ is the number of ways to tile an
  $n$-board using squares and $(1,m-1)$-fences then for $n\ge0$,
\[
A_{mn+r}^{(m)}=f_n^{m-r}f_{n+1}^r, \quad r=0,\ldots,m-1,
\]
where $f_n=F_{n+1}$.
\end{thm}
\begin{proof}
We identify the following bijection between the tilings of a
$(mn+r)$-board using squares and $(1,m-1)$-fences and the
square-domino tilings of an ordered $m$-tuple of $r$ $(n+1)$-boards
followed by $m-r$ $n$-boards. For convenience we number the boards in
this $m$-tuple from 0 to $m-1$ and the cells in the $(mn+r)$-board
from 0 to $mn+r-1$.  Tile board $j$ in the $m$-tuple with the contents
(taken in order) of the cells of the given $(mn+r)$-board fence-square
tiling whose cell number modulo $m$ is $j$. The posts of any
$(1,m-1)$-fence (which will always lie on two consecutive cells with
the same cell number modulo $m$) get mapped to a domino in board
$j$. The procedure is reversed by splicing the square-domino tilings
of the $m$-tuple of boards, hence establishing the bijection. The
number of square-domino tilings of the $m$-tuple of boards is
$f_{n+1}^rf_n^{m-r}$ and the result follows.
\end{proof}

\begin{thm}\label{T:bijJ}
If $B_n$ is the number of $n$-tilings using squares and
fences then
\begin{equation}\label{e:bijJ}
B_n=J_{n+1}.
\end{equation}
\end{thm}
\begin{proof}
As in the proof of Theorem~\ref{T:An}, we condition on
the last metatile. If the last metatile contains $m$ tiles, there are
$B_{n-m}$ possible $(n-m)$-tilings. As the three possible metatiles
contain 1, 2, and 2 tiles we have
\begin{equation}
\label{e:Bn}
B_{n}=\delta_{n,0}+B_{n-1}+2B_{n-2}, \quad B_{n<0}=0.
\end{equation}
where the $\delta_{n,0}$ is to ensure that $B_0=1$ so that when an
$n$-tiling is just one metatile we count precisely one
tiling. Comparing \eqref{e:Bn} with \eqref{e:Jn} gives the result.
\end{proof}

\section{Combinatorial proofs of identities involving the Fibonacci squares and golden rectangle numbers}

The proofs of Identities \ref{I:Sf2}, \ref{I:Sff}, and \ref{I:A2n-1}
(and of Identities \ref{I:SJ}, \ref{I:SJ2n}, \ref{I:SJ2n-1}, and
\ref{I:JJ} in the next section) follow the techniques of Benjamin and
Quinn \cite{BQ=03}.

\begin{idn}\label{I:Sf2}
For $n\ge0$,
\[
\sum_{j=0}^nf_j^2=f_nf_{n+1}.
\]
\end{idn}
\begin{proof}
How many tilings of a $(2n+1)$-board are there? Answer~1:
$A_{2n+1}=f_nf_{n+1}$. Answer~2: condition on the position of the last
metatile which is not a bifence. The tiles after this metatile must be
bifences and thus occupy $4m$ cells $(0\leq m\leq\floor{n/2})$. If the
last non-bifence metatile is a square there are
$A_{2n+1-4m-1}=f_{n-2m}^2$ tilings. If it is a filled fence there are
$A_{2n+1-4m-3}=f_{n-1-2m}^2$ tilings. Summing over all possible values
of $m$ for both cases gives the result.
\end{proof}

\begin{idn}\label{I:Sff}
For $n\ge0$,
\[
\sum_{j=0}^{2n-1}f_jf_{j+1}=f_{2n}^2-1.
\]
\end{idn}
\begin{proof}
How many tilings of a $4n$-board use at least 1 square? Answer~1:
$A_{4n}-1=f_{2n}^2-1$ since the only way to tile a $4n$-board without
using squares is to use only bifences. Answer~2: condition on the
position of the last metatile which is not a bifence. This gives
$A_{4n-4m-1}=f_{2n-2m}f_{2n-2m-1}$ tilings if the last non-bifence
metatile is a square and $A_{4n-4m-3}=f_{2n-2m-1}f_{2n-2m-2}$ if it is
a filled fence for $0\leq m<n$.  Summing over $m$ for both cases gives
the result.
\end{proof}

\begin{idn}\label{I:A2n-1}
\begin{subequations}
\label{e:A2n-1}
\begin{align}
\label{e:A2n-1e}
2\sum_{j=0}^{n-2}f_jf_{j+2}-f_{n-2}f_{n-1}
&=f_n^2-1, & n>1,\\
2\sum_{j=0}^{n-2}f_jf_{j+2}+f_{n-1}^2
&=f_nf_{n+1}-1, & n>0.
\label{e:A2n-1o}
\end{align}
\end{subequations}
\end{idn}
\begin{proof}
How many tilings of an $m$-board use at least 1 fence?
Answer~1: $A_m-1$ since
only the tiling using just squares does not use any fences.
Answer~2: condition on the location of the last fence. If the last
fence is in a metatile of length $l$ starting at cell $k+1$ (where $0\leq
k\leq m-l$) there will be $A_k$ possible tilings since the cells after
the metatile must all be occupied by squares.
 The two metatiles
containing fences are of length 3 and 4. Summing numbers of tilings
over all positions and types of the last fence-containing metatile and
equating the two answers gives
\[
2\sum_{k=0}^{m-4}A_k+A_{m-3}=A_m-1.
\]
Replacing $m$ in this by $2n$ and $2n+1$ leads to \eqref{e:A2n-1e} and
\eqref{e:A2n-1o}, respectively, after using \eqref{e:Aff} and
$f_j^2+f_jf_{j+1}=f_jf_{j+2}$.
\end{proof}

The proofs of Identities \ref{I:Sjff} and \ref{I:sumbinff} (and of
Identities \ref{I:sumkJ}, \ref{I:sumbinJ}, and \ref{I:sum2km5J} in the
next section) follow the techniques used in \cite{BCS19,EA20}. As far
as we know, all these identities are new.

\begin{idn}\label{I:Sjff} For $n\ge0$,
\[
f_nf_{n+1}=1+\floor{n/2}+\sum_{j=1}^njf_{n-j}f_{n-j+1}.
\]
\end{idn}
\begin{proof}
How many tilings of a $(2n+1)$-board contain at least two squares?
Answer~1: $A_{2n+1}-\frac12n-1$ ($A_{2n+1}-\frac12(n+1)$) if $n$ is
even (odd) since the only possible tilings with less than 2 squares
when $n$ is even (odd) is one free square (filled fence) among $n/2$
($\frac12(n-1)$) bifences and there are $\frac12n-1$ ($\frac12(n+1)$)
such tilings.  Answer~2: condition on the location of the second
square. The metatile containing this must end on an even cell,
$2j$. Written in terms of symbols (see the caption to
Fig.~\ref{f:metatiles}), the tiling of the first $2j$ cells must end
in $S$. This leaves one $S$ that may be placed anywhere among the $F$
symbols which number $j-1$. The number of ways to tile the cells to
the right of the $2j$th cell is $A_{2n+1-2j}$. Summing over all
possible $j$ gives $\sum_{j=1}^njA_{2(n-j)+1}$. Equating this to
Answer~1 and simplifying gives
\[
A_{2n+1}-\floor{n/2}-1=\sum_{j=1}^njA_{2(n-j)+1}.
\]
The identity follows from \eqref{e:A2n+1}.
\end{proof}
Note that if we consider the tilings of a $2n$-board that contain at
least two squares we obtain Identity~2.1 of \cite{EA20}.

To generalize Identity~\ref{I:Sjff} we first define $C_n^{(r)}$ as the number
of ways to tile a $(2n+1)$-board using $2r+1$ squares (and $n-r$
fences).

\begin{lemma}\label{L:Cq}
For $n\ge r\ge0$,
\begin{equation}\label{e:Cq}
C^{(r)}_n=C^{(r)}_{n-2}+\binom{n+r}{2r}, \quad C^{(r)}_{n<0}=0.
\end{equation}
\end{lemma}
\begin{proof}
In symbolic form, a tiling can end in either $S$ or $FF$. If $S$, the
number of ways to place the remaining $2r$ squares and $n-r$ fences is
$\tbinom{n+r}{2r}$. If $FF$, there are $C_{n-2}^{(r)}$ ways to place
the remaining tiles. There are no tilings if $n<0$.
\end{proof}

As will be shown in Theorem~\ref{T:Rnk}, $C^{(r)}_n$ is the $(n,r)$th
element of the $(1/[(1-x)(1-x^2)],x/(1-x)^2)$ Riordan array
(\seqnum{A158909}).

\begin{idn}\label{I:sumbinff} For $p>0$,
\[
f_nf_{n+1}=\sum_{r=0}^{p-1}C^{(r)}_n
+\sum_{j=p}^{n}\binom{j+p-1}{2p-1}f_{n-j}f_{n+1-j}.
\]
\end{idn}
\begin{proof}
How many tilings of a $(2n+1)$-board have at least $2p$ squares?
Answer~1: the total number of tilings minus the tilings
that contain less than $2p$ squares, i.e.,
\[
A_{2n+1}-\sum_{r=0}^{p-1}C^{(r)}_n.
\]
Answer~2: we condition on the location of the $2p$th square. If the
metatile containing this lies on the $2j$th cell, in the symbolic
representation, there are $2p-1$ $S$ and $j-p$ $F$ that precede the
$2p$th $S$ and hence $\binom{j+p-1}{2p-1}$ ways to arrange them. There
are $A_{2n+1-2j}$ ways to place the remaining tiles after the $2j$th
cell.  Summing over all possible $j$ and equating the result to
Answer~1 gives
\[
A_{2n+1}-\sum_{r=0}^{p-1}C^{(r)}_n=\sum_{j=p}^n\binom{j+p-1}{2p-1}A_{2(n-j)+1},
\]
and the identity follows from \eqref{e:A2n+1}.
\end{proof}

\section{Combinatorial proofs of identities involving the Jacobsthal numbers}

\begin{idn}\label{I:SJ}
For $n\ge0$,
\[
2\sum_{r=1}^nJ_r=J_{n+2}-1.
\]
\end{idn}
\begin{proof}
How many $(n+1)$-tilings use at least 1 fence? Answer~1:
$B_{n+1}-1$. Answer~2:
condition on the last metatile containing a fence. If this last
metatile contains the $(j+1)$th and $(j+2)$th tiles ($0\le j\le n-1$)
then there remain $j$ unspecified tiles. As there are two types of
metatile containing a fence there are a total of $2B_{j}$ tilings for
each $j$. Summing over $j$ we have
\[
\sum_{j=0}^{n-1} 2B_j=B_{n+1}-1,
\]
and the identity follows from Theorem~\ref{T:bijJ}.
\end{proof}

\begin{idn}\label{I:SJ2n}
For $n\ge0$,
\[
\sum_{r=1}^{2n}J_r=J_{2n+1}-1.
\]
\end{idn}
\begin{proof}
How many $2n$-tilings use at least 1 square? Answer~1:
$B_{2n}-1$ since there is only one $2n$-tiling
without a square (the all-bifence tiling). Answer~2:
condition on the last square which must be the
$2(n-m)$th tile ($0\le m\le n-1$) since any
metatiles after the last square are
bifences. If this last square is a free square there are $2n-2m-1$
remaining unspecified tiles. There are $2n-2m-2$ if it
is inside a filled fence. Summing over all possible $m$ and equating
to Answer~1 gives
\[
B_{2n}-1=\sum_{r=0}^{2n-1}B_r.
\]
The identity follows from \eqref{e:A2n+1}.
\end{proof}

\begin{idn}\label{I:SJ2n-1}
For $n\ge0$,
\[
\sum_{r=1}^{2n-1}J_r=J_{2n}.
\]
\end{idn}
\begin{proof}
The proof follows that for Identity~\ref{I:SJ2n} but we count the
number of $(2n-1)$-tilings by conditioning on the last square.
\end{proof}

\begin{idn}\label{I:JJ} For $m,n\ge0$,
\[
J_{m+n+1}=J_{m+1}J_{n+1}+2J_mJ_n.
\]
\end{idn}
\begin{proof}
The number of $(m+n)$-tilings is $B_{m+n}$. Of these there are
$B_mB_n$ tilings where the $m$th tile is the last tile in a
metatile. If the $m$th tile is the first tile in a metatile containing
two tiles, there are $m-1$ unspecified tiles before it and $n-1$
unspecified tiles after the metatile. As there are two kinds of
two-tile metatiles we have
$B_{m+n}=B_mB_n+2B_{m-1}B_{n-1}$.
The identity then follows from \eqref{e:A2n+1}.
\end{proof}

\begin{idn}\label{I:sumkJ} For $n\ge0$,
\[
J_{n+1}=\ceil{\tfrac12(n+1)}+\sum_{j=1}^{n-1}jJ_{n-j}.
\]
\end{idn}
\begin{proof}
How many $n$-tilings have at least two squares? Answer~1:
$B_n-\frac12(n+1)$ ($B_n-\frac12n-1$) when $n$ is odd (even) since the
possible tilings with one square when $n$ is odd (even) are one filled
fence (free square) placed among $\frac12(n-1)$ ($\frac12(n-2)$)
bifences and there are $\frac12(n+1)$ ($n/2$) such tilings, and the
only possible tiling with no squares is the all-bifence tiling which
only occurs when $n$ is even. Answer~2: condition on the second
metatile containing an $S$. The symbolic representation of the tiling
up to and including this must end in an $S$. If this $S$ is the $j$th
tile, there are $j-1$ ways to order the symbols preceding it and thus
$(j-1)B_{n-j}$ $n$-tilings. Summing over all possible $j$, equating to
Answer~1, and simplifying gives
\[
B_n-\ceil{\tfrac12(n+1)}=\sum_{j=2}^n(j-1)B_{n-j}.
\]
The identity is obtained on replacing $j$ by $j+1$ and using
Theorem~\ref{T:bijJ}.
\end{proof}

As before, we can generalize Identity~\ref{I:sumkJ}. We need the
following definition and lemma. Let $D^{(r)}_n$ be the number of
$n$-tilings that contain exactly $r$ squares. As the only tilings with
no squares are the all-bifence tilings, for $n>0$, $D^{(0)}_n$ is 1
(0) when $n$ is
even (odd). For convenience we make $D^{(0)}(0)=1$.
\begin{lemma}\label{L:Dq}
For $n\ge r>0$,
\begin{equation}\label{e:Dq}
D^{(r)}_n=D^{(r)}_{n-2}+\binom{n-1}{r-1}.
\end{equation}
\end{lemma}
\begin{proof}
The symbolic representation of a tiling must end in either $S$ or
$FF$. If $S$, we are free to place the remaining $n-1$ tiles (of which
$r-1$ are squares) in any order; this gives
$\binom{n-1}{r-1}$ possibilities. If $FF$, there are
$D^{(r)}_{n-2}$ ways to place the remaining tiles.
\end{proof}
As will be shown in Theorem~\ref{T:RnkB}, $D^{(r)}_n$ is the
$(1/(1-x^2),x/(1-x))$ Riordan array (\seqnum{A059260}).

\begin{idn}\label{I:sumbinJ} For $p>0$,
\[
J_{n+1}=\sum_{r=0}^{p-1}D^{(r)}_n
+\sum_{k=p}^{n}\binom{k-1}{p-1}J_{n+1-k}.
\]
\end{idn}
\begin{proof}
How many $n$-tilings have at least $p$ squares?
Answer~1: the total number of tilings minus the tilings
that contain less than $p$ squares, i.e.,
\[
B_n-\sum_{r=0}^{p-1}D^{(r)}_n.
\]
Answer~2: we condition on the location of the $p$th square. If
it is the $k$th tile, there are $\binom{k-1}{p-1}$ ways
to place the first $k$ tiles and
$B_{n-k}$ ways to place the remaining tiles. Summing over all
possible $k$ and equating the result to Answer~1 gives
\[
B_n-\sum_{r=0}^{p-1}D^{(r)}_n=\sum_{k=p}^n\binom{k-1}{p-1}B_{n-k},
\]
and the identity follows from Theorem~\ref{T:bijJ}.
\end{proof}

\begin{idn}\label{I:sum2km5J} For $n>0$,
\[
J_{n+1}=n+J_{n-1}+\sum_{k=3}^n(2k-5)J_{n+1-k}.
\]
\end{idn}
\begin{proof}
For $n>0$, how many $n$-tilings have at least two fences?
Answer~1: $B_n-1-(n-2+1)$ since only the all-square tiling and
tilings with 1 filled fence among $n-2$ squares have less than two
fences. Answer~2: condition on the location of the second fence. If it
is the $k$th tile ($k=3,\ldots,n-1$) and part of a filled fence or the
first tile in a bifence, the first fence is part of a filled fence
among $k-3$ squares and hence there are $2(k-2)B_{n-(k+1)}$ tilings
for these cases. If the second fence is the end of bifence and is the
$k$th tile ($k=2,\ldots,n$), the tiles before the bifence are all
squares and hence there are $B_{n-k}$ tilings in this case. Summing
over all possible $k$, changing $k$ to $k-1$ in the first sum, and
equating to Answer~1 gives
\[
B_n-n=2\sum_{k=4}^n(k-3)B_{n-k}+\sum_{k=2}^nB_{n-k}
=B_{n-2}+\sum_{k=3}^n(2k-5)B_{n-k}.
\]
The identity then follows from Theorem~\ref{T:bijJ}.
\end{proof}

\begin{idn}
For $n\ge0$,
\[
J_{n+1}=F_{n+1}+\sum_{j=2}^nJ_{j-1}F_{n+1-j}.
\]
\end{idn}
\begin{proof}
First note that the number of $n$-tilings with no bifences is given by
$S_n=\delta_{0,n}+S_{n-1}+S_{n-2}$ and hence $S_n=F_{n+1}$. How many
$n$-tilings have at least one bifence? Answer~1: $B_n-S_n$. Answer~2:
condition on the last bifence. When the second fence it contains is
the $j$th tile ($j=2,\ldots,n$) then the number of tilings is
$B_{j-2}S_{n-j}$. Summing over all possible $j$ and equating this to
Answer~1 gives
\[
B_n-S_n=\sum_{j=2}^nB_{j-2}S_{n-j}.
\]
The identity follows from
applying $S_n=F_{n+1}$ and Theorem~\ref{T:bijJ}.
\end{proof}

\section{A Pascal-like triangle giving the number of $n$-tilings using
  $k$ fences}

\begin{figure}
\begin{tabular}{c|*{13}{p{1.5em}}}
$n$ $\backslash$ $k$&\mbox{}\hfill0&\mbox{}\hfill1&\mbox{}\hfill2&\mbox{}\hfill3&\mbox{}\hfill4&\mbox{}\hfill5&\mbox{}\hfill6&\mbox{}\hfill7&\mbox{}\hfill8&\mbox{}\hfill9&\mbox{}\hfill10&\mbox{}\hfill11&\mbox{}\hfill12\\\hline
 0~~&\mbox{}\hfill  1 \\
 1~~&\mbox{}\hfill  1 &\mbox{}\hfill  0 \\
 2~~&\mbox{}\hfill  1 &\mbox{}\hfill  1 &\mbox{}\hfill  1 \\
 3~~&\mbox{}\hfill  1 &\mbox{}\hfill  2 &\mbox{}\hfill  2 &\mbox{}\hfill  0 \\
 4~~&\mbox{}\hfill  1 &\mbox{}\hfill  3 &\mbox{}\hfill  4 &\mbox{}\hfill  2 &\mbox{}\hfill  1 \\
 5~~&\mbox{}\hfill  1 &\mbox{}\hfill  4 &\mbox{}\hfill  7 &\mbox{}\hfill  6 &\mbox{}\hfill  3 &\mbox{}\hfill  0 \\
 6~~&\mbox{}\hfill  1 &\mbox{}\hfill  5 &\mbox{}\hfill 11 &\mbox{}\hfill 13 &\mbox{}\hfill  9 &\mbox{}\hfill  3 &\mbox{}\hfill  1 \\
 7~~&\mbox{}\hfill  1 &\mbox{}\hfill  6 &\mbox{}\hfill 16 &\mbox{}\hfill 24 &\mbox{}\hfill 22 &\mbox{}\hfill 12 &\mbox{}\hfill  4 &\mbox{}\hfill  0 \\
 8~~&\mbox{}\hfill  1 &\mbox{}\hfill  7 &\mbox{}\hfill 22 &\mbox{}\hfill 40 &\mbox{}\hfill 46 &\mbox{}\hfill 34 &\mbox{}\hfill 16 &\mbox{}\hfill  4 &\mbox{}\hfill  1 \\
 9~~&\mbox{}\hfill  1 &\mbox{}\hfill  8 &\mbox{}\hfill 29 &\mbox{}\hfill 62 &\mbox{}\hfill 86 &\mbox{}\hfill 80 &\mbox{}\hfill 50 &\mbox{}\hfill 20 &\mbox{}\hfill  5 &\mbox{}\hfill  0 \\
10~~&\mbox{}\hfill  1 &\mbox{}\hfill  9 &\mbox{}\hfill 37 &\mbox{}\hfill 91 &\mbox{}\hfill148 &\mbox{}\hfill166 &\mbox{}\hfill130 &\mbox{}\hfill 70 &\mbox{}\hfill 25 &\mbox{}\hfill  5 &\mbox{}\hfill  1 \\
11~~&\mbox{}\hfill  1 &\mbox{}\hfill 10 &\mbox{}\hfill 46 &\mbox{}\hfill128 &\mbox{}\hfill239 &\mbox{}\hfill314 &\mbox{}\hfill296 &\mbox{}\hfill200 &\mbox{}\hfill 95 &\mbox{}\hfill 30 &\mbox{}\hfill  6 &\mbox{}\hfill  0 \\
12~~&\mbox{}\hfill  1 &\mbox{}\hfill 11 &\mbox{}\hfill 56 &\mbox{}\hfill174 &\mbox{}\hfill367 &\mbox{}\hfill553 &\mbox{}\hfill610 &\mbox{}\hfill496 &\mbox{}\hfill295 &\mbox{}\hfill125 &\mbox{}\hfill 36 &\mbox{}\hfill  6 &\mbox{}\hfill  1 \\
\end{tabular}
\caption{A Pascal-like triangle with entries $\protect\tchb{n}{k}$
  (\seqnum{A059259}).}
\label{f:triB}
\end{figure}

We define $\tchb{n}{k}$ as the number of $n$-tilings which
contain exactly $k$ fences. We define $\tchb{0}{0}=1$ so that the result
\begin{equation}
\label{e:Bn=sum}
B_n=\sum_{k=0}^n\chb{n}{k}
\end{equation}
is valid for $n\geq0$.  The first 12 rows of the triangle whose
entries are $\tchb{n}{k}$ are shown in Figure~\ref{f:triB}. As will be
shown later via its connection with a Riordan array, the
triangle is sequence \seqnum{A059259}.

\begin{idn}For $n\ge0$,
\[
\chb{n}{0}=1.
\]
\end{idn}
\begin{proof}
There is only one way to tile without using any fences.
\end{proof}

\begin{idn}For $n\ge1$,
\[
\chb{n}{1}=n-1.
\]
\end{idn}
\begin{proof}
If only one of the $n$ tiles is a fence, there must be 1 filled fence and $n-2$
free squares making a total of $n-1$
metatile positions. The filled fence can be placed in any of these.
\end{proof}

The following two identities describe, respectively, the entries in
the first and second diagonal of the triangle.
\begin{idn} For $n\ge0$,
\[
\chb{n}{n}=\begin{cases}1,& \text{$n$ even};\\0,& \text{$n$ odd}.\end{cases}
\]
\end{idn}
\begin{proof}
An all-fence tiling must be composed of just
bifences. This can only occur if the number of tiles is even.
\end{proof}

\begin{idn} For $m>0$,
\[
\chb{2m-1}{2m-2}=\chb{2m}{2m-1}=m.
\]
\end{idn}
\begin{proof}
If there are $2m-1$ or $2m$ fences and 1 square, there must be $m-1$
bifences. The remaining metatile is then, respectively, a free square or
a filled fence. There are $m$ possible positions for this remaining
metatile.
\end{proof}

The following identity shows that the third diagonal of the triangle is \seqnum{A002620}.
\begin{idn} For $m>0$,
\[
\chb{2m}{2m-2}=m^2; \quad \chb{2m+1}{2m-1}=m(m+1).
\]
\end{idn}
\begin{proof}
When 2 out of $2m$ tiles are squares there must be either $m-1$
bifences and 2 free squares (totalling $m+1$ metatile positions) or
$m-2$ bifences and 2 filled fences (giving $m$ metatile
positions). There are $\binom{m+1}{2}$ places to put the squares in
the first case and $\binom{m}{2}$ ways to place the filled fences in
the second. The total number of tilings is thus
$\binom{m}{2}+\binom{m+1}{2}=m^2$.  When 2 out of $2m+1$ tiles are
squares, there must be $m-1$ bifences, 1 filled fence, and 1 free
square, and thus $m+1$ metatile positions. There are therefore
$2\binom{m+1}{2}=m(m+1)$ ways to place the free square and filled
fence.
\end{proof}

The following two identities show that the third and fourth columns
of the triangle are
\seqnum{A000124} and \seqnum{A003600}, respectively.
\begin{idn}For $n\ge2$,
\[
\chb{n}{2}=\binom{n-2}{2}+n-1.
\]
\end{idn}
\begin{proof}
If there are 2 fences, there are either 2 filled fences or 1
bifence. In the first case there are $n-4$ free squares and hence a
total of $n-2$ metatile positions in which to place the filled fences. There
are thus $\binom{n-2}{2}$ ways to place the filled fences.
In the
second case there are $n-2$ free squares and thus $n-1$ metatile
positions in which the bifence can be placed.
\end{proof}

\begin{idn}For $n\ge3$,
\[
\chb{n}{3}=\binom{n-3}{3}+2\binom{n-2}{2}.
\]
\end{idn}
\begin{proof}
If there are 3 fences, there are either 3 filled fences or 1
bifence and 1 filled fence. In the first case there are $n-6$ free
squares and 3 filled fences giving a total $n-3$ metatile positions to place
the filled fences. In the second case there are $n-4$ free squares and
thus $n-2$ metatile positions to place the filled fence and bifence.
\end{proof}

\begin{idn}\label{I:bin=sum}
For $n\ge k>0$,
\begin{equation}
\label{e:bin=sum}
\binom{n}{k}=\chb{n}{k}+\chb{n-1}{k-1}.
\end{equation}
\end{idn}
\begin{proof}
Interpret $\binom{n}{k}$ as the tilings of an $(n+k)$-board with $k$
dominoes ($D$) and $n-k$ squares ($S$). Proceeding from left to right
along the board, replace $DD$ by a bifence, $DS$ by a filled fence,
and then leave any of the remaining $S$ as they are.  Except for the
case of a `left over' single $D$
at the right end of the board, this generates all
possible $n$-tilings using $k$ fences. If the $(n+k)$-board ends in an
isolated $D$, ignore it and hence obtain a $(n-1)$-tiling with $k-1$
fences. In both cases the scheme is reversible.
\end{proof}

\begin{idn}\label{I:rrB}
For $n>k>0$,
\begin{equation}
\label{e:rrB}
\chb{n}{k}=\chb{n-1}{k}+\chb{n-1}{k-1}.
\end{equation}
\end{idn}
\begin{proof}
An $n$-tiling such that $n>k$ must contain a free square or filled
fence. Construct a bijection between $n$-tilings using $k$ fences and
$(n-1)$-tilings using $k$ or $k-1$ fences as follows. In the
$n$-tiling find the final square. If it is free, remove it to obtain
an ($n-1$)-tiling with $k$ fences. If the square is part of a filled
fence, remove the fence to obtain an ($n-1$)-tiling with $k-1$ fences.
\end{proof}

\begin{idn}\label{I:triDq}
For $n\ge r\ge0$, $\tchb{n}{n-r}=D^{(r)}_n$.
\end{idn}
\begin{proof}
The result follows from the definition of $D^{(r)}_n$ since
$\tchb{n}{n-r}$ is also the number of $n$-tilings containing $r$
squares.
\end{proof}

\begin{idn}\label{I:rr2B}
\begin{equation}
\label{e:rr2B}
\chb{n}{k}=\delta_{n,0}\delta_{k,0}+\chb{n-1}{k}+\chb{n-2}{k-1}+\chb{n-2}{k-2}.
\end{equation}
\end{idn}
\begin{proof}
We count $\tchb{n}{k}$ by conditioning on the last metatile on the
board. If the metatile contains $m$ tiles of which $j$ are fences, for
the remaining tiles the number of $(n-m)$-tilings is
$\tchb{n-m}{k-j}$. Summing these for the three types of metatile gives
the result.
\end{proof}

A $(p(x),q(x))$ Riordan array is a lower triangular matrix whose
$(n,k)$th entry is the coefficient of $x^n$ in the series for
$p(x)\{q(x)\}^k$ \cite{SGWW91}.

\begin{thm}\label{T:RnkB}
If $R(n,k)$ is the $(n,k)$th entry of the $(1/(1-x^2),x/(1-x))$
Riordan array then
\begin{equation}\label{e:Rnk}
\chb{n}{k}=R(n,n-k).
\end{equation}
\end{thm}
\begin{proof}
Let $p=1/(1-x^2)$, $q=x/(1-x)$. Then $R(n-l,k-j)$ is the coefficient
of $x^n$ in the expansion of $x^lpq^{k-j}$. Multiplying the identity
$q=x+x^2+x^2q$ by $pq^{k-1}$ and taking the coefficient of
$x^n$ gives $R(n,k)=R(n-1,k-1)+R(n-1,k-1)+R(n-2,k-1)+R(n-2,k)$
for $n>2$, $k>0$. Taking $R(n<0,k)=R(n<k,k)=0$ and including terms to
arrive at a relation that is also compatible with the values of
$R(k,n)$ for $0\le n\le 2$ and $k=0$ gives
\begin{equation}\label{e:Rrr}
R(n,k)=\delta_{n,0}\delta_{k,0}
+R(n-1,k-1)+R(n-1,k-1)+R(n-2,k-1)+R(n-2,k),
\end{equation}
which is then valid for all $n$ and $k$. Substituting \eqref{e:Rnk}
into \eqref{e:rr2B}, replacing $k$ by $n-k$, and noting that
$\delta_{n,0}\delta_{n-k,0}$ can be rewritten as $\delta_{n,0}\delta_{k,0}$,
gives \eqref{e:Rrr}.
\end{proof}

From Identity~\ref{I:triDq}, $R(n,k)=D^{(k)}_n$. In other words, a
combinatorial interpretation of $R(n,k)$ is the number of $n$-tilings
that use $k$ squares (and $n-k$ $(1,1)$-fences). Then from
Lemma~\ref{L:Dq} we have for $n\ge k\ge0$,
\begin{equation}
R(n,k)=R(n-2,k)+\binom{n-1}{k-1}.
\end{equation}
This allows us to prove a conjecture given in the OEIS entry for
\seqnum{A059259} concerning
\seqnum{A071921} which is the square array $a(n,m)$ given by
$a(0,m\ge0)=1$,
\begin{equation}\label{e:anm}
a(n,m)=\sum_{r=0}^{m-1}\binom{n-1+2r}{n-1}.
\end{equation}
Using our notation, the conjecture is as follows.
\begin{idn}
\[
\chb{n+2m}{2m}=a(n,m+1).
\]
\end{idn}
\begin{proof}
From Theorem~\ref{T:RnkB}, $\tchb{n+2m}{2m}=R(n+2m,n)$. Repeatedly
applying \eqref{e:Dq} gives
\[
R(n+2m,m)=\binom{n-1+2m}{n-1}+\binom{n-1+2(m-1)}{n-1}+\cdots
+\binom{n-1+2}{n-1}+R(n,n).
\]
Using the fact that $R(n,n)=1$ the result follows from \eqref{e:anm}.
\end{proof}

If $b$, $f$, and $s$ are, respectively,
the numbers of bifences, filled fences, and
free squares in an $n$-tiling using $k$ fences then
it is easily seen that
\begin{subequations}
\label{e:nkbfsB}
\begin{align}
\label{e:nB}
n&=2b+2f+s, \\
k&=2b+f.
\label{e:kB}
\end{align}
\end{subequations}

\begin{idn}\label{I:gfB}
\begin{equation}
\chb{n}{k}=\begin{cases}
\displaystyle\sum_{b=b\rb{min}}^{b\rb{max}} \binom{n-k+b}{k-b}\binom{k-b}{b},
&b\rb{min}\leq b\rb{max},\\
0, & b\rb{min}>b\rb{max},
\end{cases}
\label{e:gfB}
\end{equation}
where $b\rb{min}=\max(0,\ceil{k-n/2})$ and
$b\rb{max}=\floor{k/2}$.
\end{idn}
\begin{proof}
For given values of $n$ and $k$ we sum the number of tilings for all
possible values of $b$. The maximum number of
bifences $b\rb{max}$
is obtained from \eqref{e:kB} when $f$ is 0 or 1 depending on
whether $k$ is even or odd, respectively. Eliminating $f$ from
\eqref{e:nkbfsB} gives
\[
b=\tfrac12(2k-n+s).
\]
If $2k-n$ is negative, the minimum possible value of $b$ is
zero. Otherwise $b\rb{min}$ is obtained when $s$ is 0 or 1 when
$2k-n$ is even or odd, respectively.
From \eqref{e:nkbfs} we have that the total number of metatiles,
$b+f+s=n-k+b$. The number of ways of tiling using $b$ bifences, $f$
filled fences, and $s$ free squares is the multinomial coefficient
$\binom{b+f+s}{b,\,f,\,s}$ which may be re-expressed as a product of
binomial coefficients written in terms of $b$, $n$, and $k$.  There
will be no possible values of $b$ and therefore no tilings if
$b\rb{min}>b\rb{max}$.
\end{proof}

\begin{cor}
\begin{align*}
J_{n+1}&=\sum_{k=0}^n\sum_{b=\max(0,\ceil{k-n/2})}^{\floor{k/2}} \binom{n-k+b}{k-b}\binom{k-b}{b}.
\end{align*}
\end{cor}
\begin{proof}
The result follows from \eqref{e:Bn=sum}, Theorem~\ref{T:bijJ}, and
Identity~\ref{I:gfB}.
\end{proof}

The $(n,k)$th entry, which we will denote here by $\tch{n}{k}_{1/2}$,
of the Pascal-like triangle \seqnum{A123521} is the number of ways to
tile an $n$-board using $k$ $(\frac12,\frac12)$-fences and $2(n-k)$
half-squares (with the shorter sides always horizontal) \cite{EA20}.
We now show that the $\tch{n}{k}_{1/2}$ triangle can be obtained from
the $\tchb{n}{k}$ triangle by removing the odd downward diagonals of
the latter which is equivalent to the following identity.

\begin{idn}\label{I:nkhalf} For $n\ge k\ge0$,
\[
\ch{n}{k}_{1/2}=\chb{2n-k}{k}.
\]
\end{idn}
\begin{proof}
The total post length of a $(\frac12,\frac12)$-fence is 1. The entry
$\tch{n}{k}_{1/2}$ can also be viewed as counting the number of
tilings that use $k$ $(\frac12,\frac12)$-fences and $2(n-k)$
half-squares since the total length occupied by the $n$ tiles is
$k+2(n-k)\frac12=n$. The entry $\tchb{2n-k}{k}$ counts the number of tilings
using $k$ $(1,1)$-fences and $2(n-k)$ squares. This latter tiling differs
from the former only in that the tiles are twice the length.
\end{proof}

\section{A Pascal-like triangle giving the number of tilings of an
  $n$-board using $k$ fences}

We define $\tch{n}{k}$ as the number of tilings of an $n$-board which
contain exactly $k$ fences (Fig.~\ref{f:tri}).
We define $\tch{0}{0}=1$ so that the result
\begin{equation}
\label{e:An=sum}
A_n=\sum_{k=0}^n\ch{n}{k}
\end{equation}
is valid for $n\geq0$.

As a result of the following identity, the upward diagonals of the
$\tchb{n}{k}$ triangle are the rows of the $\tch{n}{k}$
triangle. Equivalently, column $k$ of the $\tch{n}{k}$ triangle is
obtained by displacing column $k$ of the $\tchb{n}{k}$ triangle
downwards by $k$ (and filling the entries above with zeros). Thus we
again obtain sequences \seqnum{A000124} and \seqnum{A003600} for the $k=2$
and $k=3$ columns, respectively (Identities \ref{I:k=2} and \ref{I:k=3}). 

\begin{idn}\label{I:ch=chb}
\[
\ch{n}{k}=\chb{n-k}{k}.
\]
\end{idn}
\begin{proof}
If a tiling contains $n-k$ tiles of which $k$ are fences, the total
length is $n$.
\end{proof}

The even rows of the triangle $\tch{n}{k}$ give the triangle
$\tch{n}{k}_{1/2}$ (defined just before Identity~\ref{I:nkhalf}).

\begin{idn}\label{I:2nk}
\[
\ch{2n}{k}=\ch{n}{k}_{1/2}.
\]
\end{idn}
\begin{proof}
The number of tilings of a $2n$-board with squares and $(1,1)$-fences is the
same as the number of tilings of an $n$-board with tiles of half the length.
\end{proof}

\begin{figure}
\begin{tabular}{c|*{13}{p{1.5em}}}
$n$ $\backslash$ $k$&\mbox{}\hfill0&\mbox{}\hfill1&\mbox{}\hfill2&\mbox{}\hfill3&\mbox{}\hfill4&\mbox{}\hfill5&\mbox{}\hfill6&\mbox{}\hfill7&\mbox{}\hfill8&\mbox{}\hfill9&\mbox{}\hfill10&\mbox{}\hfill11&\mbox{}\hfill12\\\hline
 0~~&\mbox{}\hfill  1 \\
 1~~&\mbox{}\hfill  1 &\mbox{}\hfill  0 \\
 2~~&\mbox{}\hfill  1 &\mbox{}\hfill  0 &\mbox{}\hfill  0 \\
 3~~&\mbox{}\hfill  1 &\mbox{}\hfill  1 &\mbox{}\hfill  0 &\mbox{}\hfill  0 \\
 4~~&\mbox{}\hfill  1 &\mbox{}\hfill  2 &\mbox{}\hfill  1 &\mbox{}\hfill  0 &\mbox{}\hfill  0 \\
 5~~&\mbox{}\hfill  1 &\mbox{}\hfill  3 &\mbox{}\hfill  2 &\mbox{}\hfill  0 &\mbox{}\hfill  0 &\mbox{}\hfill  0 \\
 6~~&\mbox{}\hfill  1 &\mbox{}\hfill  4 &\mbox{}\hfill  4 &\mbox{}\hfill  0 &\mbox{}\hfill  0 &\mbox{}\hfill  0 &\mbox{}\hfill  0 \\
 7~~&\mbox{}\hfill  1 &\mbox{}\hfill  5 &\mbox{}\hfill  7 &\mbox{}\hfill  2 &\mbox{}\hfill  0 &\mbox{}\hfill  0 &\mbox{}\hfill  0 &\mbox{}\hfill  0 \\
 8~~&\mbox{}\hfill  1 &\mbox{}\hfill  6 &\mbox{}\hfill 11 &\mbox{}\hfill  6 &\mbox{}\hfill  1 &\mbox{}\hfill  0 &\mbox{}\hfill  0 &\mbox{}\hfill  0 &\mbox{}\hfill  0 \\
 9~~&\mbox{}\hfill  1 &\mbox{}\hfill  7 &\mbox{}\hfill 16 &\mbox{}\hfill 13 &\mbox{}\hfill  3 &\mbox{}\hfill  0 &\mbox{}\hfill  0 &\mbox{}\hfill  0 &\mbox{}\hfill  0 &\mbox{}\hfill  0 \\
10~~&\mbox{}\hfill  1 &\mbox{}\hfill  8 &\mbox{}\hfill 22 &\mbox{}\hfill 24 &\mbox{}\hfill  9 &\mbox{}\hfill  0 &\mbox{}\hfill  0 &\mbox{}\hfill  0 &\mbox{}\hfill  0 &\mbox{}\hfill  0 &\mbox{}\hfill  0 \\
11~~&\mbox{}\hfill  1 &\mbox{}\hfill  9 &\mbox{}\hfill 29 &\mbox{}\hfill 40 &\mbox{}\hfill 22 &\mbox{}\hfill  3 &\mbox{}\hfill  0 &\mbox{}\hfill  0 &\mbox{}\hfill  0 &\mbox{}\hfill  0 &\mbox{}\hfill  0 &\mbox{}\hfill  0 \\
12~~&\mbox{}\hfill  1 &\mbox{}\hfill 10 &\mbox{}\hfill 37 &\mbox{}\hfill 62 &\mbox{}\hfill 46 &\mbox{}\hfill 12 &\mbox{}\hfill  1 &\mbox{}\hfill  0 &\mbox{}\hfill  0 &\mbox{}\hfill  0 &\mbox{}\hfill  0 &\mbox{}\hfill  0 &\mbox{}\hfill  0 \\
\end{tabular}

\caption{A Pascal-like triangle with entries $\protect\tch{n}{k}$ (\seqnum{A335964}).}
\label{f:tri}
\end{figure}

\begin{idn}For $n>0$,
\[
\ch{n}{0}=1.
\]
\end{idn}
\begin{proof}
There is only one way to tile a board without using any fences.
\end{proof}

\begin{idn}For $n>2$,
\[
\ch{n}{1}=n-2.
\]
\end{idn}
\begin{proof}
If there is only one fence, there must be 1 filled fence (and $n-3$
free squares). The filled fence can be placed in any of the $n-2$
metatile positions.
\end{proof}

\begin{idn}\label{I:k=2} For $n>3$,
\[
\ch{n}{2}=\binom{n-4}{2}+n-3.
\]
\end{idn}
\begin{proof}
If there are 2 fences, there are either 2 filled fences or 1
bifence. In the first case there are $n-6$ free squares and thus a
total of $n-4$ metatile positions to place the filled fences. There
are thus $\binom{n-4}{2}$ ways to place the filled fences.
In the
second case there are $n-4$ free squares and thus $n-3$ metatile
positions in which the bifence can be placed.
\end{proof}

\begin{idn}\label{I:k=3}For $n>5$,
\[
\ch{n}{3}=\binom{n-6}{3}+2\binom{n-5}{2}.
\]
\end{idn}
\begin{proof}
If there are 3 fences, there are either 3 filled fences or 1
bifence and a filled fence. In the first case there are $n-9$ free
squares and 3 filled fences giving a total $n-6$ metatile positions
for
the filled fences. In the second case there are $n-7$ free squares and
thus $n-5$ metatile positions for the filled fence and bifence.
\end{proof}

\begin{idn}\label{I:triCq}
For $n\ge r\ge0$, $\tch{2n+1}{n-r}=C^{(r)}_n$.
\end{idn}
\begin{proof}
The result follows from the definition of $C^{(r)}_n$ since
$\tch{2n+1}{n-r}$ is also the number of ways to tile a $(2n+1)$-board
using $2r+1$ squares.
\end{proof}

\begin{idn}\label{I:rr}
\begin{equation}
\label{e:rr}
\ch{n}{k}=\ch{n-1}{k}+\ch{n-3}{k-1}+\ch{n-4}{k-2}+\delta_{0,k}\delta_{0,n}.
\end{equation}
\end{idn}
\begin{proof}
We count $\tch{n}{k}$ by conditioning on the last metatile on the
board. If the metatile is of length $l$ and contains $j$ fences, the
number of ways to tile the remaining $n-l$ cells with $k-j$ fences is
$\tch{n-l}{k-j}$. Summing these for the three types of metatile gives
the result.
\end{proof}

To show that $C^{(r)}_n$ is a Riordan array we first need a recursion
relation that involves only the odd rows of the triangle.
\begin{idn}\label{I:rrodd}
\begin{equation}
\label{e:rrodd}
\ch{2n+1}{k}=\ch{2n-1}{k}+\ch{2n-1}{k-1}+\ch{2n-3}{k-1}
+\ch{2n-3}{k-2}-\ch{2n-5}{k-3}+\delta_{0,k}\delta_{0,n}.
\end{equation}
\end{idn}
\begin{proof}
Let $E(n,k)$ denote \eqref{e:rr}. Then $E(2n+1,k)+E(2n,k)-E(2n-1,k-1)$
gives the identity.
\end{proof}

\begin{thm}\label{T:Rnk}
If $\bar{R}(n,k)$ is the $(n,k)$th entry of the $(1/[(1-x)(1-x^2)],x/(1-x)^2)$
Riordan array then
\begin{equation}\label{e:Rbnk}
\ch{2n+1}{k}=\bar{R}(n,n-k).
\end{equation}
\end{thm}
\begin{proof}
Let $p=1/[(1-x)(1-x^2)]$, $q=x/(1-x)^2$. The recursion relation
without incorporating boundary conditions for a Riordan array with
this particular $q$ is
$\bar{R}(n,k)=\bar{R}(n-1,k)+\bar{R}(n-1,k-1)+\bar{R}(n-2,k)+\bar{R}(n-2,k-1)-\bar{R}(n-3,k)$
for $n>2$, $k>0$ \cite{EA20}. Including terms to obtain a relation
valid for all $n$ and $k$ (and taking $R(n,k)=0$ if $n<0$ or $n<k$) we
arrive at
\begin{equation}\label{e:Rbrr}
\bar{R}(n,k)=\delta_{n,0}\delta_{k,0}
+\bar{R}(n-1,k)+\bar{R}(n-1,k-1)+\bar{R}(n-2,k)+\bar{R}(n-2,k-1)-\bar{R}(n-3,k).
\end{equation}
Substituting \eqref{e:Rbnk}
into \eqref{e:rrodd}, replacing $k$ by $n-k$, and noting that
$\delta_{n,0}\delta_{n-k,0}$ can be rewritten as $\delta_{n,0}\delta_{k,0}$,
gives \eqref{e:Rbrr}.
\end{proof}

From Identity~\ref{I:triCq}, $\bar{R}(n,k)=C^{(k)}_n$. In other words, a
combinatorial interpretation of $\bar{R}(n,k)$ is the number of tilings of a
$(2n+1)$-board
that use $2k+1$ squares (and $2(n-k)$ $(1,1)$-fences). Then from
Lemma~\ref{L:Cq} we have for $n\ge k\ge0$,
\begin{equation}
\bar{R}(n,k)=\bar{R}(n-2,k)+\binom{n+k}{2k}.
\end{equation}

If $b$, $f$, and $s$ are, respectively,
the numbers of bifences, filled fences, and
free squares in a tiling of an $n$-board using $k$ fences then
it is easily seen that
\begin{subequations}
\label{e:nkbfs}
\begin{align}
\label{e:n}
n&=4b+3f+s, \\
k&=2b+f.
\label{e:k}
\end{align}
\end{subequations}

\begin{idn}For $m\geq0$,
\[
\ch{4m+p}{2m+q}=\begin{cases}
1, & p=q=0;\\
m+1, & \text{$p=1$, $q=0$ or $p=3$, $q=1$};\\
(m+1)^2,&\text{$p=2$, $q=0$};
\end{cases}
\]
and for $m>0$,
\[
\ch{4m+p}{2m+q}=\begin{cases}
m,&\text{$p=-1$, $q=-1$};\\
m(m+1),&\text{$p=0$, $q=-1$}.
\end{cases}
\]
\end{idn}
\begin{proof}
Substituting $n=4m+p$ and $k=2m+q$ into \eqref{e:nkbfs} and
eliminating $m$ gives
\[
p-2q=f+s,
\]
in which the only possible values of $f$ and $s$ are non-negative integers.
The number of bifences may then be expressed as
\begin{equation}
\label{e:bfspq}
b=m+\tfrac12(q-f).
\end{equation}
Thus when $p=q=0$ we must have $f=s=0$ which corresponds to a tiling
using $m$ bifences only. When $p=1$, $q=0$ then we must have $f=0$,
$s=1$ since \eqref{e:bfspq} would imply a non-integral number of
bifences if $f=1$, $s=0$.  With the allowed case there are $m+1$
metatile positions in which to place the free square.  When $p=3$,
$q=1$ we must have $f=1$, $s=0$ and again there are $m+1$ places for
the filled fence. For $p=2$, $q=0$, either $s=2$, $f=0$ or $s=0$,
$f=2$. In the first case there are $\binom{m+2}{2}$ ways of placing
the 2 free squares. In the second there are $m-1$ bifences and hence
$\binom{m+1}{2}$ ways of placing the 2 filled fences. Adding gives the
required result. With $p=q=-1$ the only possibility is $s=0$,
$f=1$. There are then $m-1$ bifences and hence $m$ ways to place the
filled fence. Finally, if $p=0$, $q=-1$ we must have $f=s=1$. With
$m+1$ metatile positions, there are a total of $2\binom{m+1}{2}$ ways
to place the filled fence and free square.
\end{proof}

\begin{idn}\label{I:gf}
\begin{equation}
\label{e:gf}
\ch{n}{k}=\begin{cases}
\displaystyle\sum_{b=b\rb{min}}^{b\rb{max}} \binom{n-2k+b}{k-b}\binom{k-b}{b},
&b\rb{min}\leq b\rb{max},\\
0, & b\rb{min}>b\rb{max},
\end{cases}
\end{equation}
where $b\rb{min}=\max(0,\lceil\frac12(3k-n)\rceil)$ and
$b\rb{max}=\floor{k/2}$.
\end{idn}
\begin{proof}
For given values of $n$ and $k$ we sum the number of tilings for all
possible values of $b$, the number of bifences. The maximum number of
bifences is obtained from \eqref{e:k} when $f$ is 0 or 1 depending on
whether $k$ is even or odd, respectively. Eliminating $f$ from
\eqref{e:nkbfs} gives
\[
b=\tfrac12(3k-n+s).
\]
If $3k-n$ is negative, the minimum possible value of $b$ is
zero. Otherwise $b\rb{min}$ is obtained when $s$ is 0 or 1 when
$3k-n$ is even or odd, respectively.
From \eqref{e:nkbfs} we have that the total number of metatiles,
$b+f+s=n-2k+b$. The proof is then the same as for Identity~\ref{I:gfB}.
\end{proof}

\begin{cor}
\begin{align*}
f_n^2&=\sum_{k=0}^{2n}\sum_{b=\max(0,\ceil{\frac12(3k-2n)})}^{\floor{k/2}} \binom{2n-2k+b}{k-b}\binom{k-b}{b},\\
f_nf_{n+1}&=\sum_{k=0}^{2n+1}\sum_{b=\max(0,\ceil{\frac12(3k-2n-1)})}^{\floor{k/2}} \binom{2n+1-2k+b}{k-b}\binom{k-b}{b}.
\end{align*}
\end{cor}
\begin{proof}
The result follows from \eqref{e:An=sum},
Theorem~\ref{T:A2n}, and Identity~\ref{I:gf}.
\end{proof}

\begin{idn}\label{I:2n+1kbin} For $n\ge k\ge 0$,
\[
\ch{2n+1}{k}=\sum_{j=k-m}^m\binom{n+1-j}{j}\binom{n-(k-j)}{k-j},
\]
where $m=\min(\floor{(n+1)/2},k)$.
\end{idn}
\begin{proof}
From Lemma~\ref{L:bij}, $\tch{2n+1}{k}$ is also the number of
square-domino tilings of an $(n+1)$-board and an $n$-board using $k$
dominoes in total. The number of ways to tile an $(n+1)$-board with
$j$ dominoes (and $n+1-2j$ squares) is $\binom{n+1-j}{j}$. If the
$(n+1)$-board has $j$ dominoes then the $n$-board will have $k-j$
dominoes (and $n-2(k-j)$ squares). Hence there are
$\binom{n+1-j}{j}\binom{n-(k-j)}{k-j}$ ways to tile the boards if
the $(n+1)$-board has $j$ dominoes. Evidently $j$ cannot exceed $k$ or
$\floor{(n+1)/2}$ and so $m\ge j\ge k-m$. We then sum over all
possible values of $j$.
\end{proof}

\begin{idn}\label{I:2nkbin} For $n\ge k\ge 0$,
\[
\ch{2n}{k}=\sum_{j=k-m}^m\binom{n-j}{j}\binom{n-(k-j)}{k-j},
\]
where $m=\min(\floor{n/2},k)$.
\end{idn}
\begin{proof}
The proof is analogous to that of Identity~\ref{I:2n+1kbin}.
\end{proof}
Identity~\ref{I:2nkbin} is equivalent to Identity~3.2 in
\cite{EA20}. Summing Identities \ref{I:2n+1kbin} and
\ref{I:2nkbin} over all possible $k$ will, respectively, give
alternative ways of expressing $f_nf_{n+1}$ and $f_n^2$ as double sums
of products of two binomial coefficients.

\bigskip
\hrule
\bigskip

\noindent 2010 {\it Mathematics Subject Classification}:
Primary  05A19, 11B39

\noindent \emph{Keywords: }
Fibonacci identities, Jacobsthal identities,
combinatorial proof, combinatorial identities, $n$-tiling,
Pascal-like triangle, Riordan array

\bigskip
\hrule
\bigskip

\noindent (Concerned with sequences \seqnum{A000045},
\seqnum{A000124}, \seqnum{A000930}, \seqnum{A001045},
\seqnum{A001654}, \seqnum{A002620}, \seqnum{A003269},
\seqnum{A003600}, \seqnum{A006498}, \seqnum{A007598},
\seqnum{A017817}, \seqnum{A059259}, \seqnum{A059260},
\seqnum{A071921}, \seqnum{A123521}, \seqnum{A158909}, and
\seqnum{A335964}.)

\bigskip
\hrule
\bigskip


\begin{thebibliography}{1}
\providecommand{\url}[1]{\texttt{#1}}
\providecommand{\urlprefix}{URL }

\bibitem{BCS19}
A.~T. Benjamin, J.~Crouch, and J.~A. Sellers, Unified tiling proofs of a family
  of {Fibonacci} identities, \emph{Fibonacci Quart.} \textbf{57} (2019),
  29--31.

\bibitem{BHS03}
A.~T. Benjamin, C.~R.~H. Hanusa, and F.~E. Su, Linear recurrences through
  tilings and {Markov} chains, \emph{Utilitas Math.} \textbf{64} (2003), 3--17.

\bibitem{BQ=03}
A.~T. Benjamin and J.~J. Quinn, \emph{Proofs That Really Count: The Art of
  Combinatorial Proof}, Mathematical Association of America, Washington, 2003.

\bibitem{BCCG96}
R.~C. Brigham, R.~M. Caron, P.~Z. Chinn, and R.~P. Grimaldi, A tiling scheme
  for the {Fibonacci} numbers, \emph{J. Recreational Math.} \textbf{28} (1996),
  10--16.

\bibitem{Edw08}
K.~Edwards, A {Pascal}-like triangle related to the tribonacci numbers,
  \emph{Fibonacci Quart.} \textbf{46/47} (2008/2009), 18--25.

\bibitem{EA15}
K.~Edwards and M.~A. Allen, Strongly restricted permutations and tiling with
  fences, \emph{Discrete Appl. Math.} \textbf{187} (2015), 82--90.

\bibitem{EA19}
K.~Edwards and M.~A. Allen, A new combinatorial interpretation of the
  {Fibonacci} numbers squared, \emph{Fibonacci Quart.} \textbf{57}(5) (2019),
  48--53.

\bibitem{EA20}
K.~Edwards and M.~A. Allen, A new combinatorial interpretation of the
  {Fibonacci} numbers squared. {Part II}., \emph{Fibonacci Quart.} \textbf{58}
  (2020), 169--177.

\bibitem{SGWW91}
L.~W. Shapiro, S.~Getu, W.-J. Woan, and L.~C. Woodson, The {Riordan} group,
  \emph{Discrete Appl. Math.} \textbf{34} (1991), 229--239.

\end{thebibliography}
\end{document}